\documentclass[twoside, 12pt]{amsart}
\usepackage{latexsym}
\usepackage{amssymb,amsmath,amsopn}
\usepackage[dvips]{graphicx}   
\usepackage{color,epsfig}      
\usepackage{url}

\makeatletter
\def\@settitle{\begin{center}%
  \baselineskip14\p@\relax
  \bfseries
  \uppercasenonmath\@title
  \@title
  \ifx\@subtitle\@empty\else
     \\[1ex]\uppercasenonmath\@subtitle
     \footnotesize\mdseries\@subtitle
  \fi
  \end{center}%
}
\def\subtitle#1{\gdef\@subtitle{#1}}
\def\@subtitle{}
\makeatother

\newtheorem{theorem}{Theorem}

\newtheorem{lemma}{Lemma}

\newtheorem{defi}{Definition}
\newtheorem{problem}{Problem}
\newtheorem{question}{Question}

\newtheorem{conj}{Conjecture}

\DeclareMathOperator{\conv}{conv}

\DeclareMathOperator{\vol}{vol}
\DeclareMathOperator{\area}{area}
\DeclareMathOperator{\inter}{int}

\DeclareMathOperator{\bd}{bd}

\newcommand{\R}{\mathbb{R}}

\newcommand{\Sph}{\mathbb{S}}

\renewcommand{\S}{\mathcal{S}}

\begin{document}
\title[On the convex hull function]{On the convex hull and homothetic convex hull functions of a convex body}
\subtitle{}
\author[\'A. G.Horv\'ath]{\'Akos G.Horv\'ath}
\address {\'A. G.Horv\'ath \\ Department of Geometry \\ Institute of Mathematics\\
Budapest University of Technology and Economics\\
H-1111 Budapest\\
Hungary}
\email{ghorvath@math.bme.hu}

\author[Z. L\'angi]{Zsolt L\'angi}
\address {Z. L\'angi \\ MTA-BME Morphodynamics Research Group and Department of Geometry \\ Institute of Mathematics \\
Budapest University of Technology and Economics\\
H-1111 Budapest\\
Hungary}
\email{zlangi@math.bme.hu}
\date{}


\begin{abstract}
The aim of this note is to investigate the properties of the convex hull and the homothetic convex hull functions of a convex body $K$ in Euclidean $n$-space, defined as the volume of the union of $K$ and one of its translates, and the volume of $K$ and a translate of a homothetic copy of $K$, respectively, as functions of the translation vector. In particular, we prove that the convex hull function of the body $K$ does not determine $K$. Furthermore, we prove the equivalence of the polar projection body problem raised by Petty, and a conjecture of G.Horv\'ath and L\'angi about translative constant volume property of convex bodies. We give a short proof of some theorems of Jer\'onimo-Castro about the homothetic convex hull function, and prove a homothetic variant of the translative constant volume property conjecture for $3$-dimensional convex polyhedra. We also apply our results to describe the properties of the illumination bodies of convex bodies.
\end{abstract}

\thanks{Z. L\'angi was supported by grants K119670 and BME Water Sciences \& Disaster Prevention TKP2020 IE of the National Research, Development and Innovation Fund (NRDI), by the \'UNKP-20-5 New National Excellence Program of the Ministry for Innovation and Technology, and the J\'anos Bolyai Scholarship of the Hungarian Academy of Sciences.}

\maketitle

\section{Introduction}

In this paper we denote by $\R^n$ the Euclidean $n$-space, and for any $x,y \in \R^n$, the closed segment with endpoints $x,y$ by $[x,y]$. We call a compact, convex set with nonempty interior a \emph{convex body}.
It is well known that for any $o$-symmetric body $K$ there is a unique $n$-dimensional norm whose unit ball is $K$; for any $x \in \R^n$, we denote by $||x||_K$ the norm of $x$ in this norm, and write $||x||$ for the Euclidean norm of $x$. We denote by $\vol_n(K)$ and $\conv{H}$ the $n$-dimensional volume of a body $K$ and the convex hull of the set $H$, respectively. We use the notation $\Sph^{n-1}$ for the set of unit vectors in $\R^n$, and for any unit vector $u \in \Sph^{n-1}$ and convex body $K \subset \R^n$, we let $K|u^\perp$ be the orthogonal projection of $K$ onto the hyperplane through $o$ with normal vector $u$. The notation $P+Q$ means the Minkowski sum of the bodies $P$ and $Q$.

A famous result of Meyer, Reisner and Schmuckenschl\"ager \cite{meyer-reisner-schmuckenschlager} states that if $K \subset \R^n$ is an $o$-symmetric convex body
with the property that the volume $\vol_n (K \cap (x+K))$ depends only on the Minkowski norm $||x||_K$, then $K$ is an ellipsoid.
This result is a variant of the so-called \emph{covariogram problem} of Matheron \cite{matheron}, which asks whether the function $x \mapsto \vol_n (K \cap (x+K))$  (called \emph{covariogram function}) determines the convex body $K$.

One of our main concepts is introduced in the next definition.

\begin{defi}\label{defn:convexhullfunction}
Let $K$ be an $n$-dimensional convex body and to a translation vector $t\in \mathbb{R}^n$ associate the value ${G}_K(t)=\vol\mathrm{conv}\{K\cup (K+t)\}$. The function defined in this way is called the \emph{convex hull function} associated to the body $K$.
\end{defi}

This function first appeared in the literature in a 1950 paper of F\'ary and R\'edei \cite{fary-redei}, who proved that the volume of the convex hull of two convex bodies moving at constant velocity is a convex function of time (see \cite[Satz 4]{fary-redei}, and also \cite{Ahn} for the special case of polytopes).
This statement was generalized by Rogers and Shephard \cite{rogers} for general point systems, which they called \emph{linear parameter systems}, later also called \emph{shadow systems}. The method introduced by Rogers and Shephard became an important tool in solving geometric optimization problems regarding convex bodies.

The convex hull function gave rise to a number of interesting problems, many of which are still open; for a collection of such problems see the survey paper \cite{gho-surveyonconvhullvolume} and the references therein. Nevertheless, it is an interesting fact that the `dual' of the covariogram problem, that is, the question whether
the convex hull function $G_K(t)$ determines the body $K$ or not has been asked only recently by \'A. Kurusa in a private communication. For more information on volume functions defined by convex bodies that determine the body, the interested reader is referred to the survey \cite{gho-volumefunction}.

The so-called \emph{translative constant volume property} of a convex body $K$, meaning that $G_K(x)$ depends only on $||x||_K$, was defined in \cite{gho-langi-convhull} in an investigation of some extremal properties of the convex hull function and other related volume functions. The authors of \cite{gho-langi-convhull} characterized the plane convex bodies satisfying the translative constant volume property, and conjectured that any such centrally symmetric convex body in $\R^n$ with $n \geq 3$ is an ellipsoid.

In 2015, an interesting generalization of the convex hull function was introduced by Jer\'onimo-Castro \cite{castro}, in which he replaced the translate of the convex body by a homothetic copy of the body with a fixed ratio. He used this notion to prove the homothetic version of the translative constant volume conjecture in \cite{gho-langi-convhull}, and some other results related to the homothetic version of the result of Meyer, Reisner and Schmuckenschl\"ager for intersections, proved in \cite{meyer-reisner-schmuckenschlager} as well.

Following \cite{castro}, we define the following:

\begin{defi}\label{defn:homotheticconvexhullfunction}
Let $K \subset \R^n$ be a convex body containing the origin $o$ in its interior, and $\lambda \in [0,1)$. Then the function $G_{K,\lambda} : \R^n \to \R$, defined by
$$
G_{K,\lambda}(t):=\vol_n\conv\left\{K\cup \left(\lambda K+t\right)\right\},
$$
is called the \emph{$\lambda$-homothetic convex hull function} associated to $K$.
\end{defi}

It is worth noting that this function is closely related to the so-called illumination bodies, defined in \cite{werner1} in 1994 as follows (see also \cite{werner3}).

\begin{defi}\label{defn:illuminationbodies}
Let $K \subset \R^n$ be a convex body, and let $\delta > 0$. Then the convex body
\[
K^{\delta} = \left\{ x \in \R^n : \vol_n \conv \left( K \cup \{ x \} \right) \leq \vol_n(K) + \delta \right\}
\]
is called an \emph{illumination body} associated to $K$.
\end{defi}

Indeed, the sublevel sets of the function $G_{K,0}$ clearly correspond to the illumination bodies of $K$, where the fact that these sets are convex bodies follows from the result \cite{fary-redei} of F\'ary and R\'edei already mentioned in the introduction. Furthermore, the observation that a similar statement holds for any value $0 \leq \lambda < 1$ follows from a result of Jer\'onimo-Castro \cite{castro}, which we are going to introduce in detail in Section~\ref{sec:homotheticconvexhull}.

The goal of this paper is to investigate the properties of the convex hull and $\lambda$-homothetic convex hull functions of convex bodies.

In Section~\ref{sec:convexhull} we collect our results about the convex hull functions. More specifically, we show that a convex body $K$ is characterized by its convex hull function up to translations if and only if it is centrally symmetric. A variant of this problem is the conjecture in \cite{gho-langi-convhull} about centrally symmetric convex bodies satisfying the translative constant volume property. We show that this conjecture is equivalent to the polar projection body problem introduced by Petty in \cite{petty} (see also the papers \cite{gruber} of Gruber and \cite{lutwak-affineisop} of Lutwak). Finally, we prove that for $n = 3$, any $n$-dimensional convex body satisfying the translative constant volume property and having constant brightness or width is a ball.

In Section~\ref{sec:homotheticconvexhull} we examine the properties of the homothetic convex hull functions. First, we give simple proofs of some theorems of Jer\'onimo-Castro in \cite{castro}. Finally, motivated by the proof of the polar projection body problem by Martini in \cite{martini-polarprojection} for convex polytopes, we propose a homothetic version of the translative constant volume property conjecture and prove it for $3$-dimensional convex polyhedra.

\section{The convex hull function}\label{sec:convexhull}

We start with an elementary property of the convex hull function.

\begin{lemma}\label{lem:brightness}
Let $\alpha \in \R$ be an arbitrary real number and $u \in \Sph^{n-1}$. Then we have
 \begin{equation}\label{eq:basicid}
 G_K(\alpha u)=\vol_n(K)+|\alpha|\vol_{n-1}(K|u^\perp).
 \end{equation}
Consequently, the convex hull function $G_K$ determines the volume $\vol_n(K)$ and the brightness function $u\mapsto \vol_{n-1}(K|u^\perp)$, $u\in \S^{n-1}$, of $K$, and vice versa.
\end{lemma}

\begin{proof}
The equality in (\ref{eq:basicid}) is an easy consequence of Cavalieri's principle (see e.g. \cite[Section A.5]{gardner}). The second statement follows from the definition of the brightness function and (\ref{eq:basicid}).
\end{proof}

In \cite{gardner-volcic} (see also \cite{goodey-schneider-weil}) the authors proved that, apart from parallelepipeds, no convex body is characterized by its brightness function. Our next result can be regarded as the counterpart of this result for the convex hull function. Before stating it, we remark that
for any convex body $K \subset \R^n$ and $x \in \R^n$, the convex hull functions of $K$ and $x+K$ are equal, and hence, the convex hull function of a convex body can characterize the body only up to translations.

\begin{theorem}\label{thm:characterization_ch}
A convex body $K \subset \R^n$ is characterized by its convex hull function $G_{K}$ up to translations if and only if $K$ is centrally symmetric.
\end{theorem}

\begin{proof}
By Lemma~\ref{lem:brightness}, a convex body is characterized by its convex hull function if and only if it is characterized by its brightness function and volume. Recall that the brightness function of a convex body $K$ is the support function of a convex body, called the \emph{projection body $\Pi K$} of $K$, and the family of convex bodies whose projection body is $\Pi K$ is called the \emph{projection class} of $K$ (see e.g. \cite{gardner}). Thus, the problem of finding the convex bodies $K \subset \R^n$ determined by $G_K$ is equivalent to finding the convex bodies $K \subset \R^n$ with the property that the unique convex body in the projection class of $K$ with volume equal to $\vol_n(K)$ is $K$.

On the other hand, a projection class contains exactly one $o$-symmetric convex body (called the \emph{Blaschke body} of the elements of the class), and this body has unique maximal volume in the class (see \cite{gho-volumefunction} or Theorems 4.4.3 and 3.3.9 in \cite{gardner}). Thus, every $o$-symmetric convex body is characterized by its convex hull function, implying the assertion for centrally symmetric convex bodies.

Finally, if $K$ is not centrally symmetric, then $-K$ is not a translate of $K$. As we clearly have $\Pi K = \Pi (-K)$, and $\vol_n(K) = \vol_n(-K)$, implying that $G_K = G_{-K}$, showing that if $K$ is not centrally symmetric, then its convex hull function does not characterize $K$ up to translations.
\end{proof}

We note that the first-named author in \cite[Theorem 6]{gho-volumefunction} gave an example of two convex bodies $K, L \subset \R^n$ for any $n \geq 2$ such that $G_K = G_L$, and $L$ is not the image of $K$ under any isometry of $\R^n$.

Next, we recall the notion of translative constant volume property from \cite{gho-langi-convhull}; ote that two convex bodies are said to \emph{touch each other} if they intersect and their interiors are disjoint. We remark also that the translates $x+K$ and $y+K$ of a convex body $K$ touch each other if and only if $\|x-y\|_{K-K} = 1$.

\begin{defi}\label{def:translconstvol}
If, for a convex body $K \in \R^n$, we have that $\vol_n (\conv ((v+K) \cup (w+K)))$ has the same value for any touching pair of translates,
we say that $K$ satisfies the \emph{translative constant volume property}.
\end{defi}

We note that an $o$-symmetric convex body $K$ satisfies the translative constant volume property if and only if $G_K(x)$ depends only on the norm $\|x\|_K$ of $x$. Hence, the problem of characterizing the convex bodies satisfying the translative constant volume property is the analogue for the convex hull function of the problem of characterizing the convex bodies $K$ whose covariogram function depends only on $\| x \|_K$, solved in the paper \cite{meyer-reisner-schmuckenschlager} of Meyer, {Reisner and Schmuckenschl\"ager.

We recall that a $2$-dimensional $o$-symmetric convex curve is a \emph{Radon curve}, if, for the convex hull $K$ of a suitable affine image of the curve, it holds that its polar $K^\circ$ is a rotated copy of $K$ by $\frac{\pi}{2}$ (cf. \cite{martini-swanepoel-antinorm}). It is well known that a curve is Radon if and only if in the norm induced by its convex hull, Birkhoff orthogonality is symmetric (see, e.g. \cite{gho-langi-convhull}).
The next theorem can be found in \cite{gho-langi-convhull}.

\begin{theorem}[G.Horv\'ath, L\'angi, \cite{gho-langi-convhull}]\label{thm:planartcvp}
For any plane convex body $K$, the following are equivalent.
	\begin{itemize}
		\item[(1)] $K$ satisfies the translative constant volume property.
		\item[(2)] The boundary of the central symmetral of $K$ is a Radon curve.
		\item[(3)] $K$ is a body of constant width in a Radon norm.
	\end{itemize}
\end{theorem}
Motivated by Theorem~\ref{thm:planartcvp} and the well-known fact that if every planar section of a normed space is Radon, then the unit ball of the space is an ellipsoid (cf. \cite{alonso-benitez} or \cite{martini-swanepoel-antinorm}), the authors in \cite{gho-langi-convhull} proposed Conjecture~\ref{conj:translconstvol}.

\begin{conj}\label{conj:translconstvol}
Let $n \geq 3$. Then any $o$-symmetric $n$-dimensional convex body satisfying the translative constant volume property is an ellipsoid.
\end{conj}

If $K \subset \R^n$ is a convex body, then the \emph{projection body} $\Pi K$ of $K$ is defined as the convex body whose support function is $h_{\Pi K} (u) = \vol_{n-1}(K | u^{\perp})$ for all $u \in \mathbb{S}^{n-1}$. The polar of the projection body of $K$ is called the \emph{polar projection body} of $K$, and is denoted by
$\Pi^\circ K = \left( \Pi K \right)^{\circ}$. A famous problem of convex geometry is the so-called \emph{polar projection problem} proposed by Petty in \cite{petty}, which is stated below.

\begin{conj}[Petty, \cite{petty}]\label{conj:polarprojection}
If an $o$-symmetric convex body $K \subset \R^n$ with $n \geq 3$ satisfies $\Pi^\circ K=\lambda K$ for some $\lambda \in \R$, then $K$ is an ellipsoid.
\end{conj}

We remark that in a more general, non-symmetric form, this problem asks for the characterization of the convex bodies $K$ satisfying the property that their polar projection bodies and difference bodies are similar. This version was settled by Martini in \cite{martini-polarprojection} for convex polytopes, who proved, using an elegant argument, that the only convex polytopes with the above property are the simplices.
Our next result establishes a connection between Conjecture~\ref{conj:translconstvol} and Conjecture~\ref{conj:polarprojection}.

\begin{theorem}\label{thm:conjequivpolarproblem}
Let $n \geq 3$. Then for any convex body $K \subset \R^n$ the following are equivalent.
\begin{itemize}
\item $K$ satisfies the translative constant volume property.
\item For some $\lambda > 0$, $\Pi^\circ K=\lambda (K-K)$ holds.
\end{itemize}
\end{theorem}

\begin{proof}
First, assume that $K \subset \R^n$ is an $o$-symmetric convex body.
Recall that the \emph{gauge function} of $K$ is defined as the function $\rho_K : \mathbb{S}^{n-1} \to \R$, $\rho_K(u) = \sup \{ \tau : \tau > 0, \tau u \in K\}$. The effect of polarity on the gauge and the support functions of $K$ is well known, and can be summarized by the equalities $\rho_{K^{\circ}} = \frac{1}{h_K}$
and $h_{K^{\circ}} = \frac{1}{\rho_K}$.

Since $K$ is $o$-symmetric, we clearly have that for any $u \in \mathbb{S}^{n-1}$ and $\tau > 0$, $K + \tau u$ touches $K$ if and only if $\tau = 2\rho_K(u)$.
Thus, by Lemma~\ref{lem:brightness}, if $K$ satisfies the translative constant volume property, then there is some constant $\Delta > 0$ such that for all $u \in \mathbb{S}^{n-1}$,
\begin{equation}\label{eq:reformofconstvolprop}
\Delta = 2 \rho_K(u) \vol_{n-1}(K|u^\perp) = 2 \frac{h_{\Pi K} (u)}{h_{K^{\circ}}(u)}.
\end{equation}
Clearly, (\ref{eq:reformofconstvolprop}) is equivalent to $\frac{\Delta}{2} K^{\circ} = \Pi K$ and also to $\frac{2}{\Delta} K = \Pi^{\circ} K$. This implies Theorem~\ref{thm:conjequivpolarproblem} for $o$-symmetric convex bodies. To prove it in the general case, it is sufficient to observe that for any $u \in \mathbb{S}^{n-1}$ and $\tau > 0$, $K+\tau u$ touches $K$ if and only if $\tau = \rho_{K-K}(u)$.
\end{proof}

\begin{remark}
For $o$-symmetric convex bodies there is a sharp upper bound on the constant $\lambda > 0$ in Theorem~\ref{thm:conjequivpolarproblem}.
Indeed, it was proved in \cite{martini-mustafaev} (see also \cite{gho-langi-convhull}) that
if $c^{tr}(K)$ denotes the maximum volume of the convex hull of the convex body $K$ and a translate of $K$ intersecting $K$, normalized by $\vol_n(K)$, then
\[
c^{tr}(K) \geq 1+\frac{2v_{n-1}}{v_n},
\]
with equality if and only if $K$ is an ellipsoid, where $v_i$ denotes the $i$-dimensional volume of the $i$-dimensional unit ball.
Furthermore, by Lemma~\ref{lem:brightness} and (\ref{eq:reformofconstvolprop}), if $K$ satisfies the translative constant volume property, then,
using the notation in the proof of Theorem~\ref{thm:conjequivpolarproblem}, we have $\Delta = \left( c^{tr}(K) - 1 \right) \vol_n(K)$, implying
$\lambda= \frac{2}{\Delta} \leq \frac{v_n}{v_{n-1}} \cdot \frac{1}{\vol_n(K)}$, with equality if and only if $K$ is an ellipsoid.
\end{remark}

A seminal result of Howard \cite{howard}, proving a conjecture of Nakajima \cite{nakajima}, states that any convex body in $\R^3$ having both constant width and constant brightness is a ball. Our next result is a similar statement involving the translative contant volume property.

\begin{theorem}\label{thm:translconstinplane}
If $K \subset \R^3$ is a $3$-dimensional convex body of constant brightness or of constant width, and it satisfies the translative constant volume property, then it is a ball.
\end{theorem}

\begin{proof}
By Theorem~\ref{thm:conjequivpolarproblem}, $K$ satisfies the translative constant volume property if and only if $\Pi^{\circ} K = \lambda (K-K)$ for some $\lambda > 0$. On the other hand, $K$ is of constant brightness if and only if $\Pi K$ is a Euclidean ball, which is also equivalent to the property that $\Pi^{\circ} K$ is a Euclidean ball. Similarly, $K$ is of constant width if and only if $K-K$ is a Euclidean ball. Consequently, if $K$ satisfies the translative constant volume property, then it is of constant brightness if and only if it is of constant width. Thus, Theorem~\ref{thm:translconstinplane} readily follows from the result of Howard in \cite{howard}.
\end{proof}

\begin{remark}
We note that the statement in Theorem~\ref{thm:translconstinplane} is false for plane convex bodies. Indeed, a plane convex body $K$ is of constant width if and only if it is of constant brightness, corresponding to the property that its central symmetral $\frac{1}{2}(K-K)$ is a Euclidean disk. On the other hand, since central symmetrization does not change the length of a longest chord of a convex body in any direction, we have that in this case $K$ satisfies also the translative constant volume property.
\end{remark}

\section{The homothetic convex hull function}\label{sec:homotheticconvexhull}

As it was mentioned in the introduction, a result of Meyer, Reisner and Schmunkenschl\"ager \cite{meyer-reisner-schmuckenschlager}
states that if for some $o$-symmetric convex body $K \subset \R^n$ and some $\tau > 0$, the volume $\vol_n (K \cap \{\tau K + x\})$ depends only on the Minkowski norm $\|x\|_K$, then $K$ is an ellipsoid. Motivated by this result, a similar problem was investigated by Jer\'onimo-Castro in \cite{castro}. In particular, he proved the following (cf. \cite[Theorems 1 and 2]{castro}):

\begin{theorem}[Jer\'onimo-Castro, 2015]\label{thm:castro2} Let $K\subset \R^n$ be a convex body with $o \in \inter(K)$ and let
$L \subset \R^n$ be an $o$-symmetric convex body. If there is a
number $\lambda \in (0, 1)$ such that $\vol_n \conv(K \cup \lambda (K + x))$ depends only on the Minkowski norm
$\|x\|_L$ , then $L$ is homothetic to $K$. In particular, if $\vol_n \conv(K \cup \lambda (K + x))$ is rotationally symmetric, then $K$ is a Euclidean ball.
\end{theorem}

We formulate an important tool in the proof of Theorem~\ref{thm:castro2} as Lemma~\ref{lem:lambdazero}, and give a slightly shorter proof of it than the one in \cite{castro}. This lemma establishes a connection between sublevel sets of $\lambda$-homothetic convex hull functions and illumination bodies. We use this lemma in the proof of Theorem~\ref{thm:3d}.

\begin{lemma}\label{lem:lambdazero}
Let $n \geq 2$, and $K \subset \R^n$ be a convex body with $o \in \inter(K)$. If for some $0 \leq \lambda < 1$, $L$ is a sublevel set of $G_{\lambda,K}$, then $\frac{1}{1-\lambda} L$ is an illumination body of $K$; i.e. it is a sublevel set of $G_{K,0}$, and vice versa.
\end{lemma}

\begin{proof}
Consider some point $t \in L$. An elementary computation shows that the center of homothety $\chi$ that maps $K$ into $t + \lambda K$ is $\frac{t}{1-\lambda}$.
Let $K'= \conv (K \cup (t+\lambda K)) \setminus (t+\lambda K)$. Then $\vol_n(K') = G_{\lambda,K}(t) - \lambda^n \vol_n(K)$. On the other hand, we clearly have $K' = \conv \left(K \cup \left\{ \frac{t}{1-\lambda} \right\} \right) \setminus \conv \left( (t+\lambda K) \cup \left\{ \frac{t}{1-\lambda} \right\} \right)$, which implies
$\vol_n(K') = (1-\lambda^n) G_{0,K}\left( \frac{t}{1-\lambda} \right)$. Thus,
\[
G_{0,K}\left( \frac{t}{1-\lambda} \right) = \frac{G_{\lambda,K}(t) - \lambda^n \vol_n(K)}{1-\lambda^n},
\]
which yields that $\frac{1}{1-\lambda} L$ is a sublevel set of $G_{0,K}$.
\end{proof}

Unfortunately, the witty proof of Theorem~\ref{thm:castro2} cannot be applied to settle Conjecture \ref{conj:translconstvol}. However, this result suggested the following problem, which appeared in \cite{gho-volumefunction}:

\begin{problem}\label{prob:homconvhull}
Let $0 \leq \lambda < 1$ be arbitrary. Is it true that the $\lambda$-homothetic convex hull function $G_{K,\lambda}$ of a convex body $K$ determines $K$; that is $G_{K,\lambda} = G_{L,\lambda}$ implies $K=L$?
\end{problem}

First, note that by Theorem~\ref{thm:characterization_ch} and also by \cite[Theorem 6]{gho-volumefunction}, the convex hull function of a convex body does \emph{not} determine the body. Our first result is an affirmative answer to Problem~\ref{prob:homconvhull}. Our argument yields also a short proof of Theorem~\ref{thm:castro2}, different from the one in \cite{castro}. Nevertheless, we will use the argument in \cite{castro} in the proof of Lemma~\ref{lem:lambdazero}.
Before stating our result, we remark that any convex body $L \subset \R^n$ containing $o$ in its interior is the unit ball of an asymmetric norm \cite{cobzas}, which we denote by $|| \cdot ||_L$.

\begin{theorem}
Let $0 \leq \lambda < 1$ and $n \geq 2$. Then the following holds.
\begin{itemize}
\item[(i)] If $K,L \subset \R^n$ are convex bodies satisfying $G_{K,\lambda} = G_{L,\lambda}$, then $K=L$.
\item[(ii)] If $K, L \subset \R^n$ are convex bodies and $G_{K,\lambda}(t)$ depends only on the $L$-norm $||t||_L$ of $t$, then $K$ is a homothetic copy of $L$.
\end{itemize}
\end{theorem}

We note that the case that $L$ is a Euclidean ball in (ii) yields that if $G_{K,\lambda}(t)$ is rotationally symmetric, then $K$ is a ball.

\begin{proof}
Let $K \subset \R^n$ be an arbitrary convex body, and $0 \leq \lambda < 1$. By its definition, $G_{K,\lambda} (t) \geq \vol_n (K)$, with equality if and only if
$\lambda K + t \subset K$. On the other hand, by convexity, $\lambda K + t \subset K$ is equivalent to $t \in (1-\lambda) K$. Thus, $G_{K,\lambda}$ is minimal exactly at the points of $(1-\lambda) K$, and this implies that
if $G_{K,\lambda}(t) = G_{L,\lambda}(t)$ for all $t \in \R^n$, then $(1-\lambda) K = (1-\lambda)L$, which, since $0 \leq \lambda < 1$, implies $K=L$.

Assume now that $G_{K,\lambda}(t)$ depends only on the $L$-norm $||t||_L$ of $t$ for some convex body $L$ with $o \in \inter L$. This implies that any level set of $G_{K,\lambda}$ is a homothetic copy of $\bd(L)$. Since $G_{K,\lambda}(\mu t) \leq G_{K,\lambda}(t)$ for any $0 \leq \mu \leq 1$, from this it follows that all sublevel sets of $G_{K,\lambda}$ are homothetic copies of $L$. In particular, this implies that $(1-\lambda)K=\{ t : G_{K,\lambda}(t) \leq \vol_n(K) \}$ is a homothetic copy of $L$, which implies (ii).
\end{proof}

It is worth remarking that by the result of F\'ary and R\'edei \cite{fary-redei} mentioned in the introduction, the function $G_{K,\lambda}(t)$ is a convex function of $t$, which yields that all its sublevel sets are convex bodies.

In \cite{castro}, Jer\'onimo-Castro also proved that if $K \subset \R^n$ is a convex body of class $C^2_+$, containing $o$ in its interior, and for some $0 < \lambda <1$, $G_{K,\lambda}(t)$ depends only on the $K$-norm $||t||_K$ of $t$, then $K$ is an ellipsoid. Observe that the conditions in Jer\'onimo-Castro's theorem imply that all sublevel sets of
$G_{K,\lambda}$ are positive homothetic copies of $K$. On the other hand, clearly, all sublevel sets of $G_K$ (apart from $\{ o \}$) are positive homothetic copies of $K$ if and only if there is a sublevel set of $G_K$ which is a positive homothetic copy of $K$. This leads to the following question.

\begin{question}\label{ques:homothetic}
Let $n \geq 2$ and $0 \leq \lambda < 1$. Determine the convex bodies $K \subset \R^n$, containing $o$ in their interiors, with the property that for some $\mu > 0$, $\mu K$ is a sublevel set of $G_{K,\lambda}$, or equivalently, with the property that one of the illumination bodies of $K$ is a positive homothetic copy of $K$.
\end{question}

We note that Question~\ref{ques:homothetic} stated for illumination bodies of a convex body, already appeared in the literature as part 1 of the so-called \emph{generalized homothety conjecture} of Werner and Ye in \cite{werner4}. As partial results in this direction, we mention a result of Stancu \cite{stancu}, proving that if $K$ has a $C^2_{+}$ boundary, and there is some $\delta_0 > 0$ such that for any $0 < \delta < \delta_0$, $K^{\delta}$ is homothetic to $K$, then $K$ is an ellipsoid, and also a result of Jer\'onimo-Castro who proved a similar result in the planar case using weaker boundary conditions for $K$.

Motivated by the result of Martini in \cite{martini-polarprojection} and also by the results of Werner \cite{werner2} and Mordhost and Werner \cite{mordhost} about the illumination bodies of convex polytopes, we investigate Question~\ref{ques:homothetic} for convex polytopes. Our main result is the following.

\begin{theorem}\label{thm:3d}
There is no convex polytope $P \subset \R^3$, with $o \in \inter(P)$, such that for some $0 \leq \lambda < 1$ and $\mu > 0$, $\mu P$ is a sublevel set of $G_{P,\lambda}$. Equivalently, there is no convex polytope $P \subset \R^3$ such that an illumination body of $P$ is a positive homothetic copy of $P$. \end{theorem}

We prove Theorem~\ref{thm:3d} for illumination bodies, as stated in the second form. To prove it, we start with some lemmas.

\begin{lemma}\label{lem:piecewise}
Let $n \geq 2$ and let $P \subset \R^n$ be a convex polytope and let $\delta > 0$ be arbitrary.
Then $P^{\delta}$ is a convex polytope, and:
\begin{itemize}
\item[(i)] The $(n-2)$-skeleton of $P^{\delta}$ is contained in the union of all facet hyperplanes of $P$.
\item[(ii)] For any facet hyperplane $H$ of $P$, $H \cap \bd ( P^{\delta})$ is contained in the $(n-2)$-skeleton of $P^{\delta}$.
\end{itemize}
\end{lemma}

\begin{proof}
Let $\mathcal{S}$ be a decomposition of $\bd(P)$ into $(n-1)$-dimensional simplices whose vertices are vertices of $P$. Note that for any simplex with vertices $p_1, \ldots, p_n$ in this decomposition the volume of $\conv \{ p_1, \ldots, p_n, t \}$ is the absolute value of the determinant with column vectors $p_2-p_1, \ldots, p_n-p_1, t-p_1$. Thus, if for any $t \in \R^n$, $\mathcal{S}_t$ denotes the subfamily of $\mathcal{S}$ consisting of the elements $F$ of $\mathcal{S}$ such that the closed supporting half space of $P$ whose boundary contains $F$ contains $t$, we have

\begin{equation}\label{eq:volume}
G_{0,P}(t) = \sum_{\conv \{ p_1,p_2,\ldots, p_n \} \in \mathcal{S}_t } \left| \det [p_2-p_1,p_3-p_1,\ldots, p_n-p_1,t-p_1] \right|,
\end{equation}
Thus, by the properties of determinants, $G_{0,P}$ is a piecewise linear convex function, implying that the illumination bodies $P^{\delta}$ are convex polytopes. Furthermore, any nonsmooth point of a level hypersurface of $G_{0,P}$ (i.e. any point in the $(n-2)$-skeleton of the corresponding illumination body $P^{\delta}$) lies in the affine hull of a simplex in $\mathcal{S}$. Thus, the $(n-2)$-skeleton of $P^{\delta}$ is contained in the union of the facet hyperplanes of $P$. On the other hand, if $t \in \bd (P^{\delta})$ moves in such a way that it crosses a facet hyperplane $H$ of $P$ at a point $t_0 \in \bd (\mu P)$, then $\mathcal{S}_t$ changes in such a way that the simplices in $H$ all become elements of $\mathcal{S}_t$ or they are all are removed from it, depending on the direction in which $t$ crosses $H$. Hence, it follows that $t_0$ is a nonsmooth point of $\bd (P^{\delta})$, implying that it belongs to the $(n-2)$-skeleton of $P^{\delta}$.
\end{proof}

\begin{lemma}\label{lem:polygon}
If $P \subset \R^2$ is a convex $k$-gon, and $Q$ is a convex $m$-gon such that $P \subset \inter (Q)$, every vertex of $Q$ belongs to a sideline of $P$, and every sideline of $P$ intersects $\bd(Q)$ in two vertices of $Q$, then $m \geq k$, with equality if and only if every vertex of $Q$ lies exactly on two sidelines of $P$.
\end{lemma}

\begin{proof}
By our conditions, any sideline of $P$ contains exactly two vertices of $Q$. Thus, the number of vertices of $Q$ on the sidelines of $P$, counted with multiplicity, is equal to $2k$. On the other hand, any vertex of $Q$ belongs to at most two sidelines of $P$, which yields that $Q$ has at least $k$ vertices, with equality if and only if every vertex of $Q$ is the intersection point of two sidelines of $P$.
\end{proof}

Now we prove Theorem~\ref{thm:3d}.

\begin{proof}
Without loss of generality, suppose for contradiction that $\mu P$ is an illumination body of $P$ with some $\mu > 1$.

Let $F$ be an arbitrary face of $P$, and let $H$ be the plane containing it. Then $H \cap \bd (\mu P)$ is the union of some edges of $\mu P$ by Lemma~\ref{lem:piecewise}. Assume that $F$ is a convex $k$-gon and $Q=H \cap (\mu P)$ is an $m$-gon. Then, by Lemma~\ref{lem:polygon}, $m \geq k$. Thus, let us assign each edge of $F$ to exactly one edge of $Q$ such that distinct edges of $F$ are assigned to distinct edges of $Q$. Since each edge of $P$ lies on exactly two faces of $P$, doing this procedure for all faces of $P$, in this way each edge of $P$ is assigned to exactly two edges of $\mu P$. On the other hand, by convexity, every edge of $\mu P$ belongs to at most two face planes of $P$, and hence, every edge of $\mu P$ is assigned to at most two edges of $P$. Since the numbers of edges of $P$ and $\mu P$ are clearly equal, this yields that every edge of $\mu P$ lies in exactly two face planes of $P$, and for every face plane $H$ of $P$, the number of sides of $P \cap H$ and $H \cap (\mu P)$ are equal, and every other face plane of $P$ intersects $H \cap \bd (\mu P)$ only in vertices.

It readily follows from Euler's formula that every convex polyhedron has a face with strictly less than six edges. Let $F$ be such a face, and let $E$ be an arbitrary edge of $F$. Then, by our conditions, $\mu E$ belongs to exactly two face planes $H_1$ and $H_2$ of $P$. Since the plane $H$ of $F$ separates $P \subset \mu P$ and $\mu E \subset \mu P$, it follows that $H_1$ and $H_2$ intersect the interior of $Q=H \cap (\mu P)$. Let these intersections be $E_1$ and $E_2$, respectively. Then, by our conditions, both $E_1$ and $E_2$ are diagonals of $Q$. Furthermore, since $\mu E$ is parallel to $E$ and $\mu E \subset H_1, H_2$, we have that $E_1$ and $E_2$ are parallel to $E$. Since $H_1 \cap H_2$ is the line through $\mu E$, we have that $E_1$ and $E_2$ are different.
Thus, $E_1$ and $E_2$ are disjoint diagonals of $Q$, implying that $Q$ has at least six vertices. On the other hand, the number of vertices of $Q$ is equal to the number of vertices of $F$, which implies that $F$ has at least six vertices, a contradiction.
\end{proof}

It is a natural question to examine Question~\ref{ques:homothetic} for convex polygons. We raise the following problem.

\begin{problem}\label{prob:affinelyregular}
Prove or disprove that if for some convex polygon $P \subset \R^2$ with $o \in \inter(P)$, $0 \leq \lambda < 1$ and $\mu > 0$, $\mu P$ is a sublevel set of $G_{P,\lambda}$, or equivalently, if an illumination body of $P$ is a positive homothetic copy of $P$, then $P$ is an affinely regular polygon.
\end{problem}

In the remaining part of our paper we give a partial answer to this problem. To be able to state our result, we introduce the following concept.

\begin{defi}
Let $P \subset \R^2$ be a convex $m$-gon. Let the sides of $P$ be $S_1, \ldots, S_m$ in counterclockwise order, where the indices are defined $\mod m$. Let $L_i$ be the sideline of $P$ through $S_i$. Let $p_{i,j}$ denote the intersection point of $L_i$ and $L_j$, if it exists. For any $0 \leq k,l \leq m$, the closed polygonal curve $\bigcup_{j=1}^m [p_{j-k-1,j+l}, p_{j-k,j+l+1}]$ is called the \emph{$(k,l)$-extension} of $P$.
\end{defi}

\begin{theorem}\label{thm:planar}
For any convex $m$-gon $P$ the following are equivalent:
\begin{itemize}
\item[(i)] An illumination body of $P$ is a positive homothetic copy of $P$.
\item[(ii)] A $(k,l)$-extension of $P$ is a positive homothetic copy of $\bd(P)$ containing $P$ in its interior for some $k,l \geq 1$ with $2 | (k+l)$, $k+l+1 < \frac{m}{2}$ such that the side homothetic to $S_i$ is $[p_{i-(k+l)/2-1, i+(k+l)/2},p_{i-(k+l)/2, i+(k+l)/2+1}]$.
\end{itemize}
Furthermore, in this case the $(k,l)$-extension of $P$ is a level curve of $G_{0,P}$ homothetic to $\bd P$,
and if these conditions are satisfied for the $(1,1)$-extension of $P$, then $P$ is an affinely regular polygon.
\end{theorem}

\begin{proof}

Assume that the boundary of an illumination body of $P$, i.e. a level curve $\Gamma$ of $G_{0,P}$, is a positive homotethic copy of $P$. Without loss of generality, we may assume that the center of homothety is $o$. This implies by Lemmas~\ref{lem:piecewise} and \ref{lem:polygon} that every vertex of $\Gamma$ is some point $p_{x,y}$. Let $[p_{x,y},p_{z,w}]$ be the side of $\Gamma$ which is the homothetic copy of $S_i$. Since $[p_{x,y},p_{z,w}]$ does not cross any sideline of $P$, we have that $|x-z|, |y-w| \leq 1$. On the other hand, since $[p_{x,y},p_{z,w}]$ is not contained in any sideline of $P$ by the properties of homothety, we have $x-z=y-w= \pm 1$. This implies that $\Gamma$ is the $(k,l)$-extension of $P$ for some values of $k$ and $l$. The fact that $P \subset \conv \Gamma$ follows from the elementary properties of the function $G_{0,P}$.

In the following, we set $r=k+l$, and denote by $[p_{i-r-1+j, i+j},p_{i-r+j, i+j+1}]$ the side homothetic to $S_i$ for some $1 \leq j \leq r-1$. Since the two vertices of $\Gamma$ in $L_i$ are $p_{i-r-1,i}$ and $p_{i,i+r+1}$, it follows that $L_i$ separates $P$ from exactly $r+1$ sides of $\Gamma$.
As the same holds for $L_{i-r-1}$, $\Gamma$ has more than $2(r+1)$ sides, implying that $k+l+1 = r+1< \frac{m}{2}$. By symmetry, in the remaining part we assume that $j \leq r+1-j$.

Note that the homothety that maps the side $S_i=[p_{i-1,i},p_{i,i+1}]$ of $P$ to the side $S_i'=[p_{i-r-1+j, i+j},p_{i-r+j, i+j+1}]$ of $\mu P$ maps the diagonal $[p_{i-1-j,i-j},p_{i+r-j,i+r+1-j}]$ of $P$ to the diagonal $[p_{i-r-1,i}, p_{i,i+r+1}]$ of $\mu P$.
Thus, it follows that $S_i$ is parallel to $[p_{i-1-j,i-j},p_{i+r-j,i+r+1-j}]$. Furthermore, for any point $x$ in the relative interior of $S_i'$, the segments $[x,p_{i-r-1+j,i-r+j}]$ and $[x,p_{i+j,i+j+1}]$ are contained in $\bd \conv (P \cup \{ x \})$. Since moving $x$ on $S_i'$ does not change $\area \conv (P \cup \{ x \})$, this yields that $S_i'$ (and also $S_i$) is parallel to $[p_{i-r-1+j,i-r+j},p_{i+j,i+j+1}]$, implying that $[p_{i-1-j,i-j},p_{i+r-j,i+r+1-j}]$ is parallel to $[p_{i-r-1+j,i-r+j},p_{i+j,i+j+1}]$. On the other hand, by $r+1 < \frac{m}{2}$, in the cyclic order of vertices of $P$ on $\bd (P)$ the two endpoints of $[p_{i-1-j,i-j},p_{i+r-j,i+r+1-j}]$ separate the endpoints of
$[p_{i-r-1+j,i-r+j},p_{i+j,i+j+1}]$, which yields that the two segments coincide. From this it follows that $i-j-1=i-r-1+j$, and hence, $2j=r=k+l$, $2 | (k+l)$, and $S_i'= [p_{i-(k+l)/2-1, i+(k+l)/2},p_{i-(k+l)/2, i+(k+l)/2+1}]$. This shows that (i) implies (ii). The converse statement can be proved by reversing this argument.

\begin{figure}[h]
\centering
\includegraphics[scale=0.6]{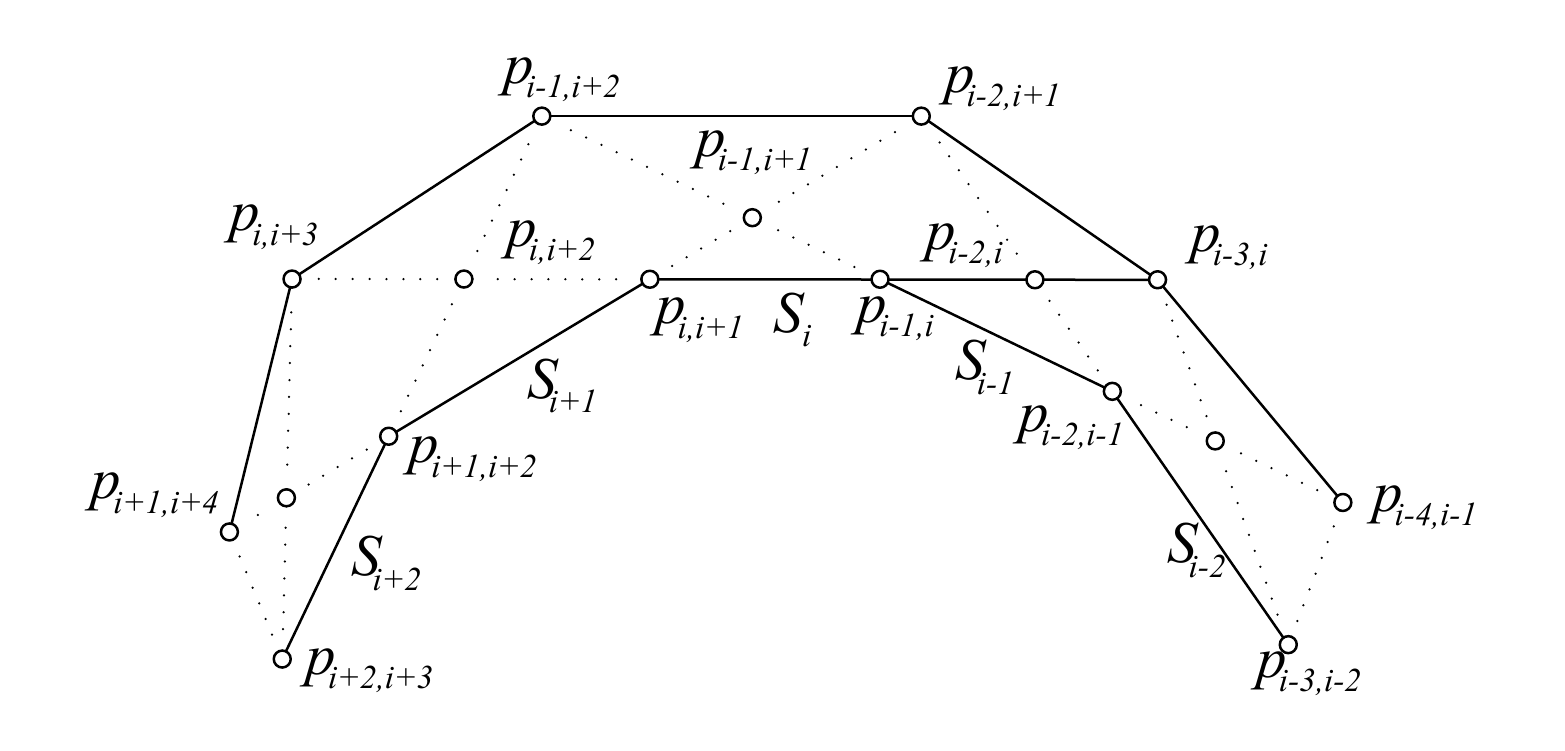}
\caption{The $(1,1)$-extension of a polygon $P$}
\label{fig:1_1}
\end{figure}

It remains to show that if the $(1,1)$-extension of $P$ satisfies the above conditions, then $P$ is affinely regular.
Thus, assume that the $(1,1)$-extension of $P$ is $\bd (\mu P)$ for some $\mu > 1$. Then, by the first part of Theorem~\ref{thm:planar}, we have that $\mu( p_{i,i+1}-p_{i-1,i}) = p_{i-1,i+2}-p_{i-2,i+1}$. This implies that $\conv \{ p_{i-1,i+1}, p_{i-1,i+2},p_{i-2,i+1} \}$ is a homothetic copy of $\conv \{ p_{i-1,i}, p_{i,i+1},p_{i-1,i+1} \}$ with homothety ratio $-\mu$ and center $p_{i-1,i+1}$ (cf. Figure~\ref{fig:1_1}). Thus, in particular, we have $p_{i-1,i+2}-p_{i-1,i+1} = \mu (p_{i-1,i+1}-p_{i-1,i})$. On the other hand, since $[p_{i-1,i+2},p_{i,i+3}]$ is parallel to $S_{i+1}$, the Intercept Theorem yields that $p_{i,i+3}-p_{i,i+1} = \mu(p_{i,i+1}-p_{i-1,i})$. We obtain similarly that $p_{i-1,i}-p_{i-3,i}=\mu (p_{i,i+1}-p_{i-1,i})$. From this, we have that $p_{i,i+3}-p_{i-3,i}= (1+2\mu) (p_{i,i+1}-p_{i-1,i})$, and hence, $p_{i+1,i+2}-p_{i-2,i-1} = \frac{1+2\mu}{\mu} (p_{i,i+1}-p_{i-1,i})$ for all values of $i$. On the other hand, it is known that if the (cyclically ordered) vertices
$q_1, \ldots, q_m$ of a polygon $Q$ satisfy the condition that $q_{i+2}-q_{i-1}= \tau (q_{i+1}-q_i)$ for some $\tau > 0$ independent of $i$, then $Q$ is an affinely regular $m$-gon (see \cite{Coxeter1}, \cite{Coxeter2}, \cite{FischerJamison}, or in a more general form, \cite{Langi_affinely}).
This implies that in this case $P$ is affinely regular.
\end{proof}

We finish the paper with the following question.

\begin{question}
What are the plane convex bodies $K \subset \R^2$ with $o \in \inter(K)$ such that a level curve of $G_{K,\lambda}$ is a Euclidean circle for some $0 \leq \lambda < 1$? Equivalently, what are the plane convex bodies $K$ such that an illumination body of $K$ is a Euclidean disk?
\end{question}

\medskip
\noindent
\textbf{Acknowledgements.}\\
The authors express their gratitude to an anonymous referee whose remarks led to the statement in Theorem~\ref{thm:characterization_ch}, and more general versions of Lemmas~\ref{lem:lambdazero} and \ref{lem:piecewise}.

\end{document}